\documentclass[12pt]{article}

\usepackage{fontenc,fontaxes}
\usepackage{graphicx,epstopdf,booktabs}
\usepackage{authblk}
\usepackage[authoryear,round]{natbib}
\usepackage{comment}
\usepackage{xspace}
\usepackage{setspace}
\usepackage[bookmarksnumbered,bookmarksopen,colorlinks,citecolor={blue}]{hyperref}
\usepackage{amsthm}
\usepackage{amsmath}
\usepackage{appendix}
\usepackage{mathtools}
\usepackage{tikz}
\usetikzlibrary{decorations.fractals,3d}
\usepackage[bf,small]{caption}
\usepackage{doi}
\usepackage{test2} 
\usepackage{paralist,subcaption}

\usepackage[margin=1.25in]{geometry}

\usepackage[shortlabels]{enumitem}
\setlist[enumerate,1]{label={\upshape(\roman*)}}

\mathtoolsset{mathic=true}

\title{Tail behavior of stopped \Levy processes with Markov modulation}

\author[1]{Brendan K. Beare}
\author[1]{Won-Ki Seo}
\author[2]{Alexis Akira Toda}
\affil[1]{School of Economics, University of Sydney}
\affil[2]{Department of Economics, University of California San Diego}

\numberwithin{equation}{section}
\numberwithin{thm}{section}

\newcommand{\Levy}{L\'evy\xspace}
\newcommand{\cadlag}{c\`adl\`ag\xspace}
\newcommand{\I}{\mathcal{I}}
\renewcommand{\S}{\mathcal{S}}

\renewcommand{\Re}{\operatorname{Re}}
\renewcommand{\Im}{\operatorname{Im}}
\renewcommand{\Pr}{\operatorname{P}}
\newcommand{\A}{\mathrm{A}}

\begin{document}
\maketitle

\begin{abstract}
This article concerns the tail probabilities of a light-tailed Markov-modulated \Levy process stopped at a state-dependent Poisson rate. The tails are shown to decay exponentially at rates given by the unique positive and negative roots of the spectral abscissa of a certain matrix-valued function. We illustrate the use of our results with an application to the stationary distribution of wealth in a simple economic model in which agents with constant absolute risk aversion are subject to random mortality and income fluctuation.
\end{abstract}

\onehalfspacing

\section{Introduction}\label{sec:intro}

This article provides machinery for characterizing the tail probabilities $\Pr(W_T>w)$ of a Markov-modulated \Levy process $(W_t)_{t\ge 0}$ stopped at a random time $T$. The class of Markov-modulated \Levy processes is a generalization of the class of \Levy processes in which the distributions of increments are governed by a Markov process $(J_t)_{t\ge 0}$, assumed here to have finite state-space. The stopping time $T$ is triggered at a Poisson rate dependent on the Markov state. We assume that our Markov-modulated \Levy process is light-tailed, in the sense that $\Pr(W_t>w)$ decays to zero exponentially or faster as $w\to\infty$ for any fixed $t$. Our main result is that $\Pr(W_T>w)$ decays to zero at an exponential rate given by the unique positive root of the spectral abscissa of a certain matrix-valued function. An analogous result holds for the lower tail.

Our article is intended as a companion piece to \citet{BeareToda2020}, where closely related results are obtained in a discrete-time setting. In both articles, the problem we solve is motivated by an influential contribution of \cite{reed2001}. Reed observed that in an economy consisting of a large number of units (e.g.\ firms, households, cities) with an exponential age distribution, each of which has been growing like a geometric Brownian motion since birth, the cross-sectional distribution of sizes is that of a geometric Brownian motion stopped at an exponentially distributed time. This turns out to be the double Pareto distribution. Reed's observation explains, at an intuitive level, why numerous heterogeneous-agent models of macroeconomic activity generate Pareto-tailed cross-sectional distributions of income, wealth, and consumption; see, e.g., \cite{Toda2014JET}, \cite{TodaWalsh2015JPE}, \cite{AlbornozFanelliHallak2016}, \cite{Arkolakis2016}, \cite{benhabib-bisin-zhu2016}, \cite{GabaixLasryLionsMoll2016}, \cite{AokiNirei2017}, \citet{CaoLuo2017}, \cite{JonesKim2018}, and \cite{KasaLei2018}.

If we take the logarithm of a random variable with a double Pareto distribution, then we obtain a random variable with a Laplace distribution. The tails of a Laplace distribution decay at an exponential rate. Consequently, the aforementioned observation of \cite{reed2001} implies that the distribution of a Brownian motion stopped at an exponentially distributed time has tails that decay at an exponential rate. The results of this article show that we continue to obtain an exponential rate of tail decay in a more general setting where in place of the Brownian motion we have any light-tailed Markov-modulated \Levy process, and where in place of the exponentially distributed stopping time we have a stopping time occurring at a state-dependent Poisson rate. This greatly expands the scope of potential applications: light-tailed \Levy processes are used widely in financial modeling as an empirically relevant generalization of Brownian motion, while Markov-modulation is a ubiquitous feature of empirically calibrated models in macroeconomics.

The remainder of this article is structured as follows. In Section \ref{sec:prelim} we briefly introduce the essential features of \Levy and Markov-modulated \Levy processes. Our primary results are presented in Section \ref{sec:results}. We illustrate the use of those results in Section \ref{sec:application} with an application to the distribution of wealth in a simple economic model with agents subject to random mortality and income fluctuation. Appendices \ref{subsec:nakagawa}-\ref{subsec:metzler} contain a presentation of mathematical results that play a key role in our proofs and may not be widely known among specialists in economic and econometric theory. Appendices \ref{subsec:proofresults} and \ref{subsec:proofapplication} contain proofs of results stated in Sections \ref{sec:results} and \ref{sec:application}.

\section{Preliminaries}\label{sec:prelim}
In this section we briefly recall the definitions of \Levy and Markov-modulated \Levy processes, and associated concepts. Our discussion is based mostly on \cite{Sato1999} and \cite{Asmussen2003}.

\subsection{\Levy processes}\label{subsec:prelimLevy}
A real-valued stochastic process $W = (W_t)_{t\ge 0}$ is called a \Levy process if it satisfies the following conditions.
\begin{enumerate}[(i)]
	\item $\Pr(W_0 = 0) = 1$.
	\item $W$ has \cadlag (right continuous with left limits) paths with probability one. 
	\item For any $s\ge 0$, $(W_{s+t}-W_s)_{t\ge 0}$ is independent of $(W_t)_{0\le t\le s}$.
	\item For any $s\ge 0$, the law of $(W_{s+t}-W_s)_{t\ge 0}$ is the same as the law of $(W_t)_{t\ge 0}$.
\end{enumerate}

The law of a \Levy process $W$ is uniquely determined by its \Levy exponent $\psi(z)=\log\E(\e^{zW_1})$, where $\log$ denotes the complex logarithm, defined here in terms of the principal value. For imaginary $z$, the \Levy exponent $\psi(z)$ is the complex logarithm of the characteristic function of $W_1$; for complex $z$ at which the expectation is well-defined, it is the complex logarithm of the Laplace transform or moment generating function (MGF) of $W_1$. Throughout this article we take the domain of $M_X(z)=\E(\e^{zX})$, the MGF of an arbitrary random variable $X$, to be the set
\begin{equation*}
\S_X = \set{z\in \C : \E(\e^{{(\Re z)X}}) < \infty},\label{eq:levydomain}
\end{equation*}
a strip in the complex plane. Every MGF is holomorphic on the interior of its domain, and has log-convex restriction to the real axis. Consequently, every \Levy exponent is holomorphic on the interior of its domain, and has convex restriction to the real axis.

A \Levy process $W$ is said to have a light upper (lower) tail if there is a positive (negative) real $s$ such that $\E(\e^{sW_1})<\infty$, meaning that the domain of the \Levy exponent $\psi(z)$ has a positive (negative) real element. Note that the property of finiteness of the MGF of $W_t$ is time-invariant, in the following sense: for any given real $s$, the three statements (a) $\E(\e^{sW_1})<\infty$, (b) $\E(\e^{sW_t})<\infty$ for some $t>0$, and (c) $\E(\e^{sW_t})<\infty$ for all $t>0$, are equivalent \citep[Theorem 25.17]{Sato1999}. Moreover, for any complex $z$ in the domain of $\psi(z)$, we have $\E(\e^{zW_t})=\e^{t\psi(z)}$ for all $t\ge 0$. A \Levy process is said to be light-tailed if it has light upper and lower tails.

Prominent examples of \Levy processes which we will make use of in our discussion include Brownian motions with drift, Poisson processes, and Cauchy processes. A Brownian motion with drift has \Levy exponent $\psi(z)=\mu z+\frac{1}{2}\sigma^2z^2$, with real parameters $\mu$ and $\sigma^2\ge 0$ called the drift and diffusion. A Poisson process has \Levy exponent $\psi(z)=\phi(\e^z-1)$, with real parameter $\phi\ge 0$ called the intensity. A Cauchy process has \Levy exponent $\psi(z)=-\abs{z}$, defined only for imaginary $z$.

\subsection{Markov-modulated \Levy processes}\label{subsec:prelimMMLP}

Let $\mathcal{N} = \set{1,\dots,N}$, a finite set. Consider a bivariate stochastic process $(W,J)=(W_t,J_t)_{t\ge 0}$ taking values in $\R\times\mathcal{N}$. We say that $W$ is a Markov-modulated \Levy process, and that $J$ is the Markov-modulator of $W$, if the following conditions are satisfied.
\begin{enumerate}[(i)]
	\item $\Pr(W_0 = 0) = 1$.
	\item $W$ and $J$ have \cadlag paths with probability one. 
	\item For any $s\ge 0$, $(W_{s+t}-W_s,J_{s+t})_{t\ge 0}$ is conditionally independent of $(W_t,J_t)_{0\le t\le s}$ given $J_s$.
	\item For any $n\in\mathcal{N}$ and any $s\ge 0$, the law of $(W_{s+t}-W_s,J_{s+t})_{t\ge 0}$ conditional on $J_s=n$ is the same as the law of $(W_t,J_t)_{t\ge 0}$ conditional on $J_0=n$.
\end{enumerate}

It is clear from conditions (iii) and (iv) that the Markov-modulator $J$ is a time-homogeneous Markov process taking values in $\mathcal{N}$. When $\mathcal{N}$ is a singleton, conditions (i)--(iv) reduce to their analogues in Section \ref{subsec:prelimLevy}, indicating that $W$ is a \Levy process. When $\mathcal{N}$ has multiple elements, $W$ is not in general a \Levy process but may be regarded, loosely speaking, as behaving like a \Levy process in a Markov environment.

The joint law of a Markov-modulated \Levy process $W$ and its Markov-modulator $J$ may be specified in terms of the following four objects.
\begin{enumerate}[(i)]
	\item \emph{Initial state probabilities}: An $N\times 1$ vector $\varpi$ with nonnegative entries $\varpi_n$ summing to one.
	\item \emph{Infinitesimal generator}: An $N\times N$ matrix $\Pi$ with nonnegative off-diagonal entries $\pi_{nn'}$, and all rows summing to zero.
	\item \emph{State-dependent \Levy exponents}: An $N\times N$ diagonal matrix-valued function $\Psi(z)$ with \Levy exponents $\psi_n(z)$ along the diagonal. (We take the domain of $\Psi(z)$ to be the intersection of the domains of the diagonal components.)
	\item \emph{MGFs for additional jumps}: An $N\times N$ matrix-valued function $\Upsilon(z)$ with diagonal entries equal to one and MGFs $\upsilon_{nn'}(z)$ in the off-diagonal entries. (We set $\upsilon_{nn'}(z)=1$ if the corresponding transition rate $\pi_{nn'}=0$, and take the domain of $\Upsilon(z)$ to be the intersection of the domains of the off-diagonal components.)
\end{enumerate}

These four objects uniquely determine the law of $(W,J)$ in the following way. First, the law of the time-homogeneous Markov process $J$ is uniquely determined by the law of $J_0$, which is given by $\varpi$, and the infinitesimal generator matrix $\Pi$. Specifically, we may draw $J_0$ according to the probabilities $\Pr(J_0=n)=\varpi_n$, and then generate $(J_t)_{t>0}$ according to the transition rates
\begin{equation*}
\lim_{t\downarrow 0}\frac{1}{t}\Pr(J_{s+t}=n'\mid J_s=n)=\pi_{nn'},\quad n\ne n'.
\end{equation*}
Next, the law of the Markov-modulated \Levy process $W$ conditional on $J$ is uniquely determined by $\Psi(z)$ and $\Upsilon(z)$ as follows. Let $t_0=0$, and let $t_1,t_2,\dotsc$ denote the successive times at which $J$ transitions between states. Given $J$, the Markov-modulated \Levy process $W$ is initialized at $W_0=0$ and evolves as a \Levy process with \Levy exponent $\psi_{n}(z)$ over each interval $(t_k,t_{k+1})$ with $J_{t_k}=n$. The transition points $t_k$ may trigger additional jumps in $W$: if $J$ transitions from state $n$ to state $n'$ at time $t_k$, then $W$ experiences a jump at time $t_k$ with size determined by the MGF $\upsilon_{nn'}(z)$. Note that the MGF $\upsilon_{nn'}(z)$ assigns probability one to a jump of size zero if $\upsilon_{nn'}(z)=1$.

If a state $n$ is never reached with probability one, then the \Levy exponent $\psi_n(z)$ and the transition densities $\pi_{nn'}$ and jump MGFs $\upsilon_{nn'}(z)$ will be undetermined for each $n'\ne n$. In this case we should simply reformulate the state space to exclude the redundant state $n$.

Let $\odot$ denote the Hadamard (entry-wise) product between two matrices of the same size. The following characterization of the conditional MGFs of a Markov-modulated \Levy process is a key input to our results.

\begin{prop}\label{prop:asmussen}
	Let $W$ be a Markov-modulated \Levy process parametrized by $\varpi$, $\Pi$, $\Psi(z)$ and $\Upsilon(z)$ as above. For any $t>0$, any $z$ in the domains of $\Psi(z)$ and $\Upsilon(z)$, and any $n,n'\in\mathcal{N}$, the conditional expectation  $\E(\e^{zW_t}1(J_t=n')\mid J_0=n)$ is equal to the $(n,n')$-entry of the matrix $\exp(t(\Psi(z)+\Pi\odot\Upsilon(z)))$.
\end{prop}
\begin{proof}
	See Proposition 2.2 of \citet[p.~311]{Asmussen2003}.
\end{proof}

For readers seeking additional details on Markov-modulated \Levy processes, a useful textbook treatment may be found in \citet[ch.~XI]{Asmussen2003}, where such processes are called Markov additive processes.

\section{Results}\label{sec:results}

Let $W=(W_t)_{t\ge 0}$ be a Markov-modulated \Levy process, with Markov-modulator $J=(J_t)_{t\ge 0}$ taking values in a finite set $\mathcal{N}=\set{1,\dots,N}$. Let $\varpi$, $\Pi$, $\Psi(z)$ and $\Upsilon(z)$ be, respectively, the $N\times1$ vector of initial probabilities, $N\times N$ infinitesimal generator matrix, $N\times N$ diagonal matrix of \Levy exponents, and $N\times N$ matrix of jump MGFs for $(W,J)$, as described in Section \ref{subsec:prelimMMLP}. These four objects uniquely determine the law of $(W,J)$. We seek to characterize the tails of $W_T$, where $T$ is a random time.

We would like to interpret $T$ as a time-of-death when mortality occurs at a state-dependent Poisson rate. To this end, we define a second Markov-modulated \Levy process $V=(V_t)_{t\ge 0}$, which shares the Markov-modulator $J$ of $W$. As a counterpart to $\Psi(z)$, we introduce an $N\times N$ diagonal matrix $\Phi$ with nonnegative real entries $\phi_n$ along the diagonal. As in Section \ref{subsec:prelimMMLP}, let $t_0=0$, and let $t_1,t_2,\dotsc$ denote the successive times at which $J$ transitions between states. We suppose that $V$ is initialized at $V_0=0$ and evolves as a Poisson process of intensity $\phi_{n}$ over each interval $(t_k,t_{k+1})$ with $J_{t_k}=n$. We suppose that $V$ is continuous at each transition time $t_k$ with probability one; therefore, the implicit counterpart to $\Upsilon(z)$ is an $N\times N$ matrix of ones. This uniquely specifies the law of $V$ given $J$. To complete our specification of the law of $(V,W,J)$, we assume that $V$ is conditionally independent of $W$ given $J$. We set $T=\inf\set{t\ge 0:V_t>0}$, the time at which $V$ experiences its first jump. A condition on $\Pi$ and $\Phi$ that is necessary and sufficient to have $\Pr(T<\infty)=1$ is given in Proposition \ref{prop:finiteT} below.

Let $\S$ be the intersection of the domains of $\Psi(z)$ and $\Upsilon(z)$; that is, the intersection of the domains of all of the \Levy exponents $\psi_n(z)$ and MGFs $\upsilon_{nn'}(z)$. Recalling our discussion of such domains in Section \ref{subsec:prelimLevy}, we observe that $\S$ is a strip in the complex plane, and that its real part $\I=\set{\Re z:z\in\S}$ is an interval, possibly a singleton or unbounded, containing zero. In order to apply Theorems \ref{thm:tailprob} and \ref{thm:exptail} below to the upper (lower) tail of $W_T$, we will require that $\I$ contains a positive (negative) element. In this sense, $W$ is required to be light-tailed; this extends the notion of a light-tailed \Levy process given in Section \ref{subsec:prelimLevy} to a Markov-modulated setting.

For $z\in\S$, define the complex matrix-valued function
\begin{equation}
\A(z)=\Psi(z)+\Pi\odot\Upsilon(z)-\Phi.\label{eq:defA}
\end{equation}
Our characterization of the tails of $W_T$ will depend on the behavior of the spectral abscissa of $\A(s)$ as a function of the real variable $s\in\I$. Recall that the spectral abscissa of a complex square matrix $A$, which we denote $\zeta(A)$, is the maximum of the real parts of the complex eigenvalues of $A$; that is,
\begin{equation*}
\zeta(A)=\max\set{\Re\lambda:\text{$\lambda$ is an eigenvalue of $A$}}.
\end{equation*}

\begin{rem}\label{rem:metzler}
	A real square matrix is said to be Metzler if its off-diagonal entries are nonnegative. As discussed in Appendix \ref{subsec:metzler}, the spectral abscissa of a Metzler matrix is an eigenvalue of that matrix, and has associated left and right eigenvectors with nonnegative entries. This result resembles the Perron-Frobenius theorem for nonnegative matrices. The matrix $\A(s)$ is Metzler for every $s\in\I$.
\end{rem}

Our first result concerns the shape of the spectral abscissa of $\A(s)$ as a function of the real variable $s\in\I$. The proof, which relies on a result of \citet{Nussbaum1986} on the convexity of spectral abscissae, may be found in Appendix \ref{subsec:proofresults}, along with the proofs of all other numbered propositions and theorems stated in this section.

\begin{prop}\label{prop:convex}
	The spectral abscissa $\zeta(\A(s))$ is a convex function of $s\in\I$ and satisfies $\zeta(\A(0))\le 0$. If $\zeta(\A(0))<0$, then the equation $\zeta(\A(s))=0$ admits at most one positive solution $s=\alpha\in\I$ and at most one negative solution $s=-\beta\in\I$.
\end{prop}

We illustrate the content of Proposition \ref{prop:convex} in Figure \ref{fig:alphabeta}. In the case depicted, the endpoints of $\I$ are finite, but these may also be infinite. The condition $\zeta(\A(0))<0$ is satisfied; note that $\A(0)=\Pi-\Phi$. The following result shows that this condition has a natural interpretation.

\begin{figure}
	\centering
	\begin{tikzpicture}[x=0.5cm,y=0.5cm,z=0.3cm,scale=6]
	\draw[thick, ->] (xyz cs:x=-1.35) -- (xyz cs:x=1.5) node[below] {$s$};
	\draw[thick, ->] (xyz cs:y=-.5) -- (xyz cs:y=1.5) node[right] {$\zeta(\A(s))$};
	\node[below left] at (xyz cs:x=-.64) {$-\beta$};
	\node[below right] at (xyz cs:x=.47) {$\alpha$};
	\node[below right] at (xyz cs:x=0) {$0$};
	\draw[smooth,samples=400,variable=\x,domain={-.94}:{1.12}] plot (\x,{-.3+ln((e^(.5*\x))/(1-((1/1.3)-(1/1))*\x-(1/(1*1.3))*\x*\x))},0);
	\draw[dashed] (xyz cs:x=-1.15,y=0) -- (xyz cs:x=-1.15,y=1.5);
	\draw[dashed] (xyz cs:x=1.3,y=0) -- (xyz cs:x=1.3,y=1.5);
	\filldraw (xyz cs:x=-1.15) circle (.3pt);
	\filldraw (xyz cs:x=1.3) circle (.3pt);
	\filldraw (xyz cs:x=-.64) circle (.3pt);
	\filldraw (xyz cs:x=.47) circle (.3pt);
	\filldraw (xyz cs:x=0) circle (.3pt);
	\draw[<-] (xyz cs:x=-1.15,y=-.05) -- (xyz cs:x=-1.15,y=-.25);
	\node[below right,font=\small, align=left] at (xyz cs:x=-1.85,y=-.25) {Left endpoint of $\I$};
	\draw[<-] (xyz cs:x=1.3,y=-.05) -- (xyz cs:x=1.3,y=-.25);
	\node[below left,font=\small, align=right] at (xyz cs:x=2,y=-.25) {Right endpoint of $\I$};
	\end{tikzpicture}
	\caption{Determination of $\alpha$ and $\beta$ from the spectral abscissa of $\A(s)$.}\label{fig:alphabeta}
\end{figure}
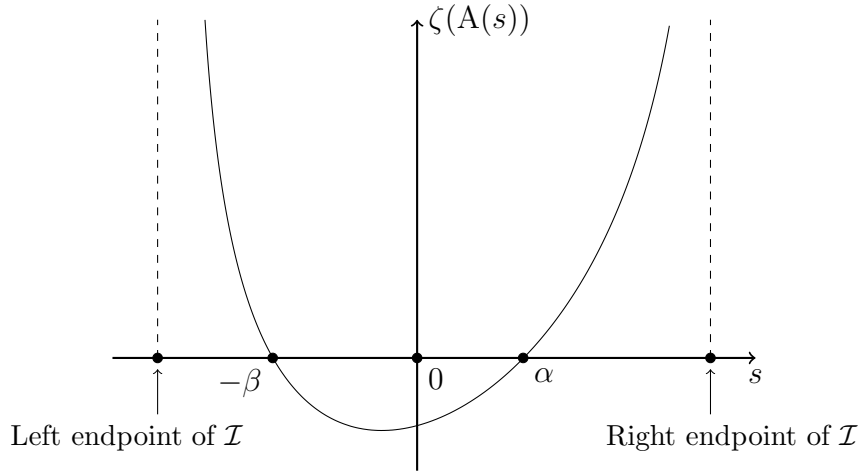

\begin{prop}\label{prop:finiteT}
	 We have $\Pr(T<\infty)=1$ if and only if $\zeta(\A(0))<0$.
\end{prop}

Our main results, Theorems \ref{thm:tailprob} and \ref{thm:exptail}, concern the case where there exists a positive and/or negative $s$ in the interior of $\I$ solving $\zeta(\A(s))=0$. Figure \ref{fig:alphabeta} depicts a case where both positive and negative solutions exist, but this need not always occur. We now give three examples of ways in which there might not exist a positive or negative solution. In each example we have $N=1$, so $W$ is a \Levy process with no Markov-modulation, and the equation $\zeta(\A(s))=0$ reduces to $\psi(s)=\phi$.

\begin{exmp}\label{ex:Cauchy}
	If $W$ is a Cauchy process then $\I=\set{0}$ and thus there cannot exist any positive or negative $s\in\I$. The problem is that a Cauchy process is not light-tailed.
\end{exmp}

\begin{exmp}\label{ex:nonneg}
	If $W$ is a Poisson process with intensity $\gamma>0$ then its \Levy exponent satisfies $\psi(s)=\gamma(\e^s-1)<0$ for all $s<0$, so there cannot exist a negative solution to $\psi(s)=\phi$. The problem is that a Poisson process is nonnegative. If $\phi>0$ then the unique positive solution is $s=\log(1+\phi/\gamma)$.
\end{exmp}

\begin{exmp}\label{ex:essentialsingularity}
	This example is more subtle. Fix $c>0$ and define a measure $\nu$ on $\R$ by $\nu(\diff x)=1(x\ge1)cx^{-2}\e^{-x}\diff x$. Let $W$ be a pure jump process with \Levy measure $\nu$, meaning that $W$ has \Levy exponent satisfying
	\begin{equation*}
	\psi(s)=\int_{-\infty}^\infty(\e^{sx}-1)\nu(\diff x)=c\int_1^\infty(\e^{(s-1)x}-\e^{-x})x^{-2}\diff x
	\end{equation*}
	for all real $s$ in its domain. The last integral is finite for $s\le1$ and infinite for $s>1$, so $\I=(-\infty,1]$. If $\phi>0$ then we may choose $c$ small enough that $\psi(1)<\phi$, in which case there will not exist a positive $s\in\I$ solving $\psi(s)=\phi$. The problem is that $\int(\e^{sx}-1)\nu(\diff x)$ jumps from below $\phi$ to infinity at the right endpoint of $\I$.
\end{exmp}

\begin{rem}\label{rem:existence}
A simple sufficient condition for the existence of a positive $s$ in the interior of $\I$ solving $\zeta(\A(s))=0$ is that (i) $\Pr(T<\infty)=1$, and (ii) at least one of the \Levy exponents $\psi_n(s)$, viewed as a function of real $s\in\I$, diverges to infinity as $s$ increases to the (positive, possibly infinite) right endpoint of $\I$. In this case $\zeta(\A(s))$ is negative for $s=0$ (Proposition \ref{prop:finiteT}) and increases to infinity as $s$ approaches the right endpoint of $\I$, since
\begin{equation*}
\zeta(\A(s))\ge\zeta(\diag(\A(s)))=\max\set{\psi_n(s)+\pi_{nn}-\phi_n:n\in\mathcal{N}}\to\infty,
\end{equation*}
using a monotonicity property of the spectral abscissa (Theorem \ref{thm:strictmonotonicity}) to obtain the inequality. Since $\zeta(\A(s))$ is convex (Proposition \ref{prop:convex}), hence continuous, on the interior of $\mathcal I$, it follows from the intermediate value theorem that there exists a positive $s$ solving $\zeta(\A(s))=0$. A similar remark applies to the existence of a negative solution.
\end{rem}

When $\Pr(T<\infty)=1$ the random variable $W_T$ is well-defined on a set of probability one. The remainder of our results concern its distribution. We begin by identifying the MGF of $W_T$. Define
\begin{equation*}
\I_-=\set{s\in \I:\zeta(\A(s))<0},\quad \S_-=\set{z\in\C:\Re z\in\I_-}.
\end{equation*}
We will obtain an expression for the MGF of $W_T$ on $\S_-$. The following property of $\S_-$ is important because it ensures that $\A(z)$ is invertible on $\S_-$.

\begin{prop}\label{prop:Sminus}
	We have $\zeta(\A(z))<0$ for every $z\in\S_-$.
\end{prop}

Let $1_N$ denote an $N\times1$ vector with each entry equal to one. We have the following representation of the MGF of $W_T$ on $\S_-$.

\begin{prop}\label{prop:MGF}
If $\Pr(T<\infty)=1$ then, for each $z\in\S_-$, the MGF of $W_T$ satisfies
\begin{equation}	
M_{W_T}(z)=\E(\e^{zW_T})=\varpi^\top\A(z)^{-1}\A(0)1_N.\label{eq:M_W_T}
\end{equation}
\end{prop}

\begin{exmp}\label{exmp:Brownian}
	(Brownian motion with drift.) Consider the simple case where there is no Markov-modulation ($N=1$), and where $W$ is a Brownian motion with drift, with \Levy exponent $\psi(z)=\mu z+\frac{1}{2}\sigma^2z^2$.  We then have $\zeta(\A(z))=\mu z+\frac{1}{2}\sigma^2z^2-\phi$, a quadratic polynomial in $z$ with positive and negative roots
	\begin{equation*}
	\alpha=\frac{-\mu+\sqrt{\mu^2+2\sigma^2\phi}}{\sigma^2},\quad-\beta=\frac{-\mu-\sqrt{\mu^2+2\sigma^2\phi}}{\sigma^2}.
	\end{equation*}
	Applying Proposition \ref{prop:MGF}, we find that the MGF of $W_T$ is given by
	\begin{equation*}
	M_{W_T}(z)=\frac{\phi}{\phi-\mu z-\frac{1}{2}\sigma^2z^2}=\frac{\alpha\beta}{(\alpha-z)(\beta+z)}.
	\end{equation*}
	This is the MGF of the asymmetric Laplace (or double exponential) distribution, which has probability density function
	\begin{equation*}
	f(x)=\begin{cases*}
	\frac{\alpha\beta}{\alpha+\beta}\e^{-\alpha x}& for $x\ge 0$,\\
	\frac{\alpha\beta}{\alpha+\beta}\e^{\beta x}& for $x<0$.
	\end{cases*}
	\end{equation*}
	The fact that a Brownian motion with drift stopped at an exponentially distributed time has the asymmetric Laplace distribution was shown by \citet{reed2001}.
\end{exmp}

In Example \ref{exmp:Brownian}, the distribution of $W_T$ has exponentially decaying tail probabilities, with the rates of exponential decay in each tail given by the positive and negative roots $\alpha$ and $-\beta$ of $\zeta(\A(s))$, as in Figure \ref{fig:alphabeta}. Our main result, Theorem \ref{thm:tailprob}, shows that this is a general feature of light-tailed Markov-modulated \Levy processes stopped at a state-dependent Poisson rate. To establish it, we rely on an assumption that $\Pi$ is irreducible. A complex $N\times N$ matrix $A=(a_{nn'})$ is said to be irreducible if, for any nonempty proper subset $\mathcal M$ of $\mathcal{N}$, there is an $n\in\mathcal M$ and an $n'\in\mathcal{N}\setminus\mathcal M$ such that $a_{nn'}\ne 0$. For the infinitesimal generator matrix $\Pi$, irreducibility means that the state variable $J_t$ cannot become permanently confined to any strict subset of the state space.

\begin{thm}\label{thm:tailprob}
	Suppose that $\Pr(T<\infty)=1$ and that $\Pi$ is irreducible. If the equation $\zeta(\A(s))=0$ admits a positive solution $s=\alpha$ in the interior of $\I$ then
	\begin{equation}\label{eq:mainup}
	0<\liminf_{w\to\infty}\e^{\alpha w}\Pr(W_T>w)\le\limsup_{w\to\infty}\e^{\alpha w}\Pr(W_T>w)<\infty.
	\end{equation}
	Similarly, if the equation $\zeta(\A(s))=0$ admits a negative solution $s=-\beta$ in the interior of $\I$ then
	\begin{equation}\label{eq:mainlow}
	0<\liminf_{w\to\infty}\e^{\beta w}\Pr(W_T<-w)\le\limsup_{w\to\infty}\e^{\beta w}\Pr(W_T<-w)<\infty.
	\end{equation}
\end{thm}

\begin{rem}
Our proof of Theorem \ref{thm:tailprob} involves showing that $\alpha$ and $-\beta$ are poles of $\A(z)^{-1}$. Under the irreducibility condition on $\Pi$, we show that these poles of $\A(z)^{-1}$ are simple, and are also simple poles of $M_{W_T}(z)$. This allows us to apply a theorem of \citet{Nakagawa2007}, discussed in Appendix \ref{subsec:nakagawa}, to obtain positive and finite lower and upper bounds for the limits inferior and superior in \eqref{eq:mainup} and \eqref{eq:mainlow}.
\end{rem}

\begin{rem}\label{rem:loglim}
	If \eqref{eq:mainup} is satisfied then there exist $w_0,c,c'\in(0,\infty)$ such that
	$$c\le\e^{\alpha w}\Pr(W_T>w)\le c'$$
	for all $w\ge w_0$. Taking logarithms, dividing by $w$, and letting $w\to\infty$, we obtain
	\begin{equation}\label{eq:loglim}
	\lim_{w\to\infty}\frac{1}{w}\log \Pr(W_T>w)=-\alpha.
	\end{equation}
	Thus \eqref{eq:mainup} implies \eqref{eq:loglim}. A similar statement applies to the lower tail.
\end{rem}

We illustrate the use of Theorem \ref{thm:tailprob} with applications to a two-state Brownian motion with drift, and a two-state linear trend.

\begin{exmp}\label{exmp:twostateBrownian}
	(Two-state Brownian motion with drift.) Suppose next that there are two Markov states ($N=2$), and that in state $n$, $W$ evolves like a Brownian motion with drift $\mu_n$ and diffusion $\sigma_n^2>0$, and $V$ evolves like a Poisson process with intensity $\phi_n$. Suppose that transitions do not trigger jumps, so that all entries of $\Upsilon(z)$ are equal to one. The rows of the infinitesimal generator $\Pi$ must sum to zero, so we write it as
	\begin{equation*}
	\Pi=\begin{bmatrix}
	-\pi_1&\pi_1\\
	\pi_2&-\pi_2
	\end{bmatrix}.
	\end{equation*}
	Suppose that $\pi_1,\pi_2>0$, so that $\Pi$ is irreducible, and suppose that $\phi_1>0$ or $\phi_2>0$, so that $\Pr(T<\infty)=1$. The matrix $\A(s)$ in \eqref{eq:defA} is given by
	\begin{equation*}
	\A(s)=\begin{bmatrix}
	\frac{1}{2}\sigma_1^2s^2+\mu_1s-\phi_1-\pi_1&\pi_1\\
	\pi_2&\frac{1}{2}\sigma_2^2s^2+\mu_2s-\phi_2-\pi_2
	\end{bmatrix}\eqqcolon\begin{bmatrix}
	a&b\\c&d
	\end{bmatrix}.
	\end{equation*}
    It has two real eigenvalues given by the solutions to the quadratic equation
	\begin{equation*}    
	\lambda^2-(a+d)\lambda+(ad-bc)=0,
	\end{equation*}
    the larger of which is the spectral abscissa
    \begin{equation}\label{eq:2stateabscissa}
    \zeta(\A(s))=\frac{a+d+\sqrt{(a+d)^2-4(ad-bc)}}{2}.
    \end{equation}
    Note that the discriminant satisfies
    $$(a+d)^2-4(ad-bc)=(a-d)^2+4bc>0$$
    due to the fact that $bc=\pi_1\pi_2>0$.
    
    To apply Theorem \ref{thm:tailprob}, we seek positive and negative values of $s$ such that the right-hand side of \eqref{eq:2stateabscissa} is equal to zero. This occurs when $a+d\le 0$ and $ad-bc=0$; that is, when
    \begin{equation}\label{eq:2statequad}
    \frac{1}{2}(\sigma_1^2+\sigma_2^2)s^2+(\mu_1+\mu_2)s-(\phi_1+\phi_2)-(\pi_1+\pi_2)\le 0
    \end{equation}
    and
    \begin{equation}\label{eq:2statequart}
    \left(\frac{1}{2}\sigma_1^2s^2+\mu_1s-\phi_1-\pi_1\right)\left(\frac{1}{2}\sigma_2^2s^2+\mu_2s-\phi_2-\pi_2\right)-\pi_1\pi_2=0.
    \end{equation}
    Denote by $f(s)$ the left-hand side of \eqref{eq:2statequart}, a quartic polynomial in $s$, and denote by $g_1(s)$ and $g_2(s)$ the two quadratic polynomials in parentheses on the left-hand side of \eqref{eq:2statequart}. Observe that
    \begin{equation}\label{eq:2stateq1}
    f(0)=\pi_1\phi_2+\phi_1\pi_2+\phi_1\phi_2>0
    \end{equation}
    since by assumption $\pi_1,\pi_2>0$ and $\phi_1>0$ or $\phi_2>0$. Each quadratic polynomial $g_n(s)$ has positive and negative roots $\alpha_n,-\beta_n$, at which our quartic polynomial $f(s)$ satisfies
    \begin{equation}\label{eq:2stateq2}
    f(-\beta_1)=f(-\beta_2)=f(\alpha_1)=f(\alpha_2)=-\pi_1\pi_2<0.
    \end{equation}
    From \eqref{eq:2stateq1}, \eqref{eq:2stateq2}, $f(\pm\infty)=\infty$, and the intermediate value theorem, we deduce that each of the four disjoint intervals
    \begin{equation}\label{eq:quarticroots}
    (-\infty,-\max\set{\beta_1,\beta_2}),\,(-\min\set{\beta_1,\beta_2},0),\,(0,\min\set{\alpha_1,\alpha_2}),\,(\max\set{\alpha_1,\alpha_2},\infty)
    \end{equation}
    contains at least one root of $f(s)$. Since $f$ is a quartic polynomial, by the fundamental theorem of algebra, $f$ has exactly four complex roots. Therefore $f$ has four distinct real roots, one in each interval in \eqref{eq:quarticroots}. The two quadratic polynomials $g_1(s)$ and $g_2(s)$ are negative over the interval $(0,\min\set{\alpha_1,\alpha_2})$, implying that their sum is negative, and so the quadratic inequality \eqref{eq:2statequad} is satisfied. On the other hand, $g_1(s)$ and $g_2(s)$ are both positive over the interval $(\max\set{\alpha_1,\alpha_2},\infty)$, implying that their sum is positive, and so \eqref{eq:2statequad} is not satisfied. Thus the unique positive root of the quartic polynomial $f(s)$ belonging to the interval $(0,\min\set{\alpha_1,\alpha_2})$ is the unique positive solution $s=\alpha$ to the equation $\zeta(\A(s))=0$ and, by Theorem \ref{thm:tailprob}, is equal to the rate of exponential decay in the upper tail of $W_T$. Similarly, the unique negative root of $f(s)$ belonging to the interval $(-\min\set{\beta_1,\beta_2},0)$ is the unique negative solution $s=-\beta$ to the equation $\zeta(\A(s))=0$, and provides the rate of exponential decay in the lower tail of $W_T$.
\end{exmp}

\begin{exmp}\label{exmp:twostatelinear}
	(Two-state linear trend.) Suppose we eliminate the diffusive component in the two-state Brownian motion considered in Example \ref{exmp:twostateBrownian} by setting $\sigma_1^2=\sigma_2^2=0$, thereby obtaining a two-state linear trend. In this case the quadratic inequality \eqref{eq:2statequad} reduces to
	\begin{equation}
	(\mu_1+\mu_2)s\le\phi_1+\phi_2+\pi_1+\pi_2,\label{eq:twostatelinear0}
	\end{equation}
	and the quartic equation \eqref{eq:2statequart} reduces to
	\begin{equation}
	\left(\mu_1s-\phi_1-\pi_1\right)\left(\mu_2s-\phi_2-\pi_2\right)-\pi_1\pi_2=0.\label{eq:twostatelinear}
	\end{equation}
	Denote by $g(s)$ the left-hand side of \eqref{eq:twostatelinear}, a quadratic polynomial in $s$ if $\mu_1\mu_2\ne 0$, or linear otherwise. Note that $g(0)>0$ by the same argument as in Example \ref{exmp:twostateBrownian}. Excluding the case $\mu_1\mu_2=0$ and assuming without loss of generality that $\mu_1\le\mu_2$, there are three cases to consider.
	\begin{case}[$\mu_1<0<\mu_2$]
		Since $g(0)>0$ and $g(\pm\infty)=-\infty$, the quadratic polynomial $g(s)$ has unique positive and negative roots $\alpha$ and $-\beta$. Thus, by Theorem \ref{thm:tailprob}, $W_T$ has exponential upper and lower tails with decay rates $\alpha$ and $\beta$ respectively.
	\end{case}
	\begin{case}[$0<\mu_1\le\mu_2$]
		Let $\alpha_n=(\pi_n+\phi_n)/\mu_n>0$, so that $g(\alpha_1)=g(\alpha_2)=-\pi_1\pi_2<0$. Since $g(0)>0$ and $g(\infty)=\infty$, the disjoint intervals $(0,\min\set{\alpha_1,\alpha_2})$ and $(\max\set{\alpha_1,\alpha_2},\infty)$ must each contain a root of the quadratic polynomial $g(s)$. The larger root violates inequality \eqref{eq:twostatelinear0}. Letting $\alpha$ denote the smaller root, we deduce from Theorem \ref{thm:tailprob} that $W_T$ has an exponential upper tail with decay rate $\alpha$.
	\end{case}
	\begin{case}[$\mu_1\le\mu_2<0$]
		Symmetrically to Case 2, $W_T$ has an exponential lower tail with decay rate given by the maximum of the two negative roots of $g(s)$.
	\end{case}
\end{exmp}

\begin{rem}
	\cite{CaoLuo2017} study a general equilibrium neoclassical growth model in continuous-time with persistent heterogeneous returns to investment. Productivity in their model is a two-state Markov process. This leads to a stationary distribution of log-wealth whose upper tail is determined in a manner similar to that of $W_T$ in Example \ref{exmp:twostatelinear} above. Theorem 3 of \cite{CaoLuo2017} establishes that the upper tail exponent of their stationary distribution of wealth is given by the solution to a quadratic equation similar to \eqref{eq:twostatelinear} above.
\end{rem}

Theorem \ref{thm:tailprob} establishes conditions under which $W_T$ has upper tail decaying exponentially at rate $\alpha$ in the sense of \eqref{eq:mainup}, and lower tail decaying exponentially at rate $\beta$ in the sense of \eqref{eq:mainlow}. It may be reasonable to ask whether \eqref{eq:mainup} and \eqref{eq:mainlow} may be strengthened to
\begin{align}
0&<\liminf_{w\to\infty}\e^{\alpha w}\Pr(W_T>w)=\limsup_{w\to\infty}\e^{\alpha w}\Pr(W_T>w)<\infty,\quad\text{and}\label{eq:exactup}\\
0&<\liminf_{w\to\infty}\e^{\beta w}\Pr(W_T<-w)=\limsup_{w\to\infty}\e^{\beta w}\Pr(W_T<-w)<\infty,\label{eq:exactlow}
\end{align}
respectively, so that $\e^{\alpha w}\Pr(W_T>w)$ and $\e^{\beta w}\Pr(W_T<-w)$ converge to positive and finite limits as $w\to\infty$. In fact, this is not possible in general: Example \ref{exmp:poisson} below provides a case where $W_T$ satisfies \eqref{eq:mainup} but not \eqref{eq:exactup}.

If we would like to strengthen Theorem \ref{thm:tailprob} such that the tails of $W_T$ decay exponentially in the sense of \eqref{eq:exactup} and \eqref{eq:exactlow}, then it suffices to impose a non-lattice condition. We say that an MGF is lattice if it corresponds to a distribution that is supported on a lattice; that is, a distribution that assigns probability one to a set of the form $\set{a+bm:m\in\Z}$ for some real $a$ and $b>0$. Otherwise, we say that an MGF is non-lattice.

\begin{thm}\label{thm:exptail}
	Suppose that either \begin{inparaenum}[\normalfont(i)]\item for some $n\in\mathcal{N}$, the MGF $\exp\psi_n(z)$ is non-lattice, or \item for some $n,n'\in\mathcal{N}$ with $\pi_{nn'}>0$, the MGF $\upsilon_{nn'}(z)$ is non-lattice. \end{inparaenum} Then Theorem \ref{thm:tailprob} remains valid with \eqref{eq:exactup} in place of \eqref{eq:mainup}, and \eqref{eq:exactlow} in place of \eqref{eq:mainlow}.
\end{thm}

\begin{rem}
	Our proof of Theorem \ref{thm:exptail} involves showing that our non-lattice condition ensures that the simple poles $\alpha$ and $-\beta$ of the MGF $M_{W_T}(z)$ in \eqref{eq:M_W_T} are the unique singularities on the lines $\Re z=\alpha,-\beta$, called the axes of convergence of $M_{W_T}(z)$. This strengthens the lower and upper bounds on the limits inferior and superior in \eqref{eq:mainup} and \eqref{eq:mainlow} obtained using the theorem of \citet{Nakagawa2007}, such that the lower and upper bounds are equal.
\end{rem}

In situations where Theorem \ref{thm:tailprob} applies but the non-lattice condition in Theorem \ref{thm:exptail} is not satisfied, \eqref{eq:exactup} or \eqref{eq:exactlow} may fail to hold, as the following example shows.

\begin{exmp}\label{exmp:poisson}
	(Poisson process.) Consider the case where there is no Markov-modulation ($N=1$), and where $W$ is a Poisson process with intensity $\gamma>0$, as in Example \ref{ex:nonneg}. We saw there that when $\phi>0$, the unique positive solution to $\psi(s)=\phi$ is $s=\log(1+\phi/\gamma)$. Therefore, applying Theorem \ref{thm:tailprob}, we find that the upper tail of $W_T$ decays exponentially in the sense of \eqref{eq:mainup}, with decay rate $\alpha=\log(1+\phi/\gamma)$.
	
	Does the upper tail of $W_T$ decay exponentially in the stronger sense of \eqref{eq:exactup}? We cannot deduce this from Theorem \ref{thm:exptail}, because the MGF $\exp\psi(z)$ corresponds to a Poisson distribution, which is supported on the lattice of integers. In fact, the upper tail of $W_T$ does not satisfy \eqref{eq:exactup}, as we now show. Using the formula for the probability mass function of the Poisson distribution, for any nonnegative integer $k$ we obtain
	\begin{equation*}	
	\Pr(W_T = k) = \E(\Pr(W_T= k\mid T))= \int_{0}^\infty \frac{(\gamma x)^k}{k!} \e^{-\gamma x}\phi\e^{-\phi x} \diff x=\frac{\phi \gamma^k}{(\gamma + \phi)^{k+1}}.
	\end{equation*}
	It follows that
	\begin{equation*}
	\Pr(W_T > w) = \Pr(W_T \ge \floor{w} + 1) = \frac{\phi}{\gamma+\phi} \sum_{k=\floor{w} + 1}^\infty \left(\frac{\gamma}{\gamma+\phi}\right)^k =\left(\frac{\gamma}{\gamma+\phi}\right)^{\floor{w}+1},
	\end{equation*}
	where $\floor{\cdot}$ rounds down to the nearest integer. Since $\e^\alpha=(\gamma+\phi)/\gamma$, we deduce that
	\begin{equation*}	
	\e^{\alpha w}\Pr(W_T > w) =\left(\frac{\gamma}{\gamma+\phi}\right)^{\floor{w}-w+1},
	\end{equation*}
	which oscillates between 1 and $\gamma/(\gamma+\phi)$ as $w\to\infty$, in contravention of \eqref{eq:exactup}.
\end{exmp}	

\section{Application to the distribution of wealth}\label{sec:application}

We illustrate our results with an application to the distribution of wealth in a simple economic model in which agents with constant absolute risk aversion (CARA) are subject to idiosyncratic income fluctuation. The model we consider is similar to those of \citet{Wang2003} and \citet{Moll2020}. Throughout this section, to avoid repeated qualifications about sets of measure zero, we regard any two stochastic processes whose paths are equal with probability one to be the same object.

\subsection{Single agent problem}\label{sec:singleagent}

Let $J=(J_t)_{t\ge 0}$ be a \cadlag Markov process taking values in the finite set $\mathcal{N}=\set{1,\dots,N}$, with irreducible $N\times N$ infinitesimal generator matrix $\Pi$, and $N\times 1$ vector of initial probabilities $\varpi$, which we take to be given by the stationary distribution of $J$. Let $t_0=0$, and let $t_1,t_2,\dotsc$ denote the successive times at which $J$ transitions between states. We suppose that for $t$ belonging to an interval $[t_k,t_{k+1})$ with $J_t=n$, an agent receives a constant flow of income $Y_t=y_n\in\R$. The agent chooses a consumption flow $C=(C_t)_{t\ge 0}$, which is a stochastic process satisfying two conditions:
\begin{enumerate}[(i)]
	\item $C$ is adapted to the filtration generated by $J$.
	\item $C$ is \cadlag, with discontinuities permitted only at $t_1,t_2,\dotsc$.
\end{enumerate}
We denote by $\mathcal{C}$ the collection of all stochastic processes $C$ satisfying (i) and (ii). Condition (i) ensures that the agent's consumption at time $t$ depends only on information available at time $t$. Condition (ii) allows us to define the agent's wealth over each interval $[t_k,t_{k+1})$ to be the unique solution to an ordinary differential equation (ODE) satisfying a boundary condition at $t_k$. Specifically, we define the agent's wealth $W^C=(W^C_t)_{t\ge 0}$, whose dependence on $C$ we make explicit in our notation, to be the unique continuous stochastic process initialized at $W^C_0=0$ which, over each interval $[t_k,t_{k+1})$, satisfies the ODE
\begin{equation}\label{eq:ODE}
	\frac{\diff W^C_t}{\diff t}=rW^C_t+Y_t-C_t,
\end{equation}
where $r>0$ is a fixed rate of interest. The existence of a unique solution to this ODE satisfying a boundary condition at $t_k$ is ensured by the fact that $Y_t-C_t$ is continuous on $[t_k,t_{k+1})$ with finite limit at $t_{k+1}$ \citep[Corollary 5.1, p.~31]{Hartman1982}.

The agent's objective function is the expected discounted flow of future utility, defined for $C\in\mathcal{C}$ by
\begin{equation}\label{eq:objective}
	U(C)=\E_0\int_0^\infty \e^{-\rho t}u(C_t)\diff t.
\end{equation}
Here, $\E_0$ denotes expectation conditional on $J_0$, the parameter $\rho>0$ is the agent's rate of discounting, and $u(\cdot)$ is the CARA utility function $u(c)=-\e^{-\gamma c}/\gamma$, with coefficient of absolute risk aversion $\gamma>0$. Since $C$ is measurable under condition (ii) and $u(\cdot)$ is bounded from above, the integral in \eqref{eq:objective} is well-defined, though possibly equal to minus infinity.

The agent's consumption flow $C$ is required to satisfy the no-Ponzi condition
\begin{equation}\label{eq:budget}
	\liminf_{t\to\infty}\e^{-\rho t}\E_0(\e^{-\gamma rW^C_t})=0.
\end{equation}
The no-Ponzi condition serves as a borrowing constraint by limiting the agent's ability to accumulate debt. We will see in Appendix \ref{subsec:proofapplication} that it corresponds to the transversality condition for the agent's optimization problem. We denote by $\mathcal{C}_1$ the collection of all $C\in\mathcal{C}$ satisfying \eqref{eq:budget}.

The optimization problem we have just described resembles a continuous-time reformulation of the discrete-time problem studied by \citet{Wang2003}. The one substantive difference is that \citet{Wang2003} assumes an autoregressive law of motion for income, which we are unable to accommodate due to our assumption that the number of Markov states is finite. Our optimization problem is also very similar to the leading example discussed by \citet{Moll2020}, also in continuous-time, which assumes a two-state income process. Those authors use a fixed lower bound for wealth as a borrowing constraint; we instead follow \citet{Wang2003} and use the transversality condition as a borrowing constraint. The optimal consumption flow is given by the following result, proved in Appendix \ref{subsec:proofapplication}.

\begin{prop}\label{prop:optimum}
	$U(C)$ is uniquely maximized over $\mathcal{C}_1$ by a consumption flow $C^\ast$, with associated wealth process $W^\ast$, such that
	\begin{equation}\label{eq:optimalc}
		C^\ast_t=rW^\ast_t+b_n\quad\text{while }J_t=n,
	\end{equation}
	where $b_1,\dots,b_N$ are the unique real numbers solving the system of equations
	\begin{equation}\label{eq:bnsystem}
		b_n=y_n-\frac{1}{\gamma}+\frac{\rho}{r\gamma}-\frac{1}{\gamma r}\sum_{n'=1}^N\pi_{nn'}\e^{\gamma(b_{n}-b_{n'})},\quad n\in\mathcal{N}.
	\end{equation}
\end{prop}

We see from Proposition \ref{prop:optimum} that the agent consumes the interest earned on wealth, plus a quantity determined by the current income level. Combining \eqref{eq:ODE} and \eqref{eq:optimalc} reveals that over each interval $(t_k,t_{k+1})$ with $J_{t_k}=n$, the agent's wealth has slope $y_n-b_n$. Wealth is therefore a piecewise linear trend with slope determined by the current Markov state. This is a simple example of a Markov modulated \Levy process; we studied the two-state case in Example \ref{exmp:twostatelinear}.

\subsection{Partial and general equilibrium}\label{sec:equilibrium}

We now characterize the distribution of wealth in a partial or general equilibrium model with a unit mass of agents. We adopt the perpetual youth formulation of birth and death introduced by \citet{Yaari1965}, in which agents die at a constant Poisson rate $\phi>0$, returning their wealth or debt to a central bank, and are reborn with zero initial wealth and initial income state drawn from the stationary distribution $\varpi$. (The central bank may alternatively be viewed as an insurance company buying and selling actuarial notes, as in \citealp[p.~140]{Yaari1965}.) The stationary age distribution under this scheme is exponential with parameter $\phi$, which is the same as the distribution of an individual agent's lifespan $T$. Agents maximize their expected discounted flow of future utility. Since the mortality rate is constant, an agent's lifespan $T$ is independent of their income flow, so that with discount rate $\tilde{\rho}>0$ we have
\begin{equation*}
\E_0\int_0^T\e^{-\tilde{\rho}t}u(C_t)\diff t=\E_0\int_0^\infty\e^{-(\tilde{\rho}+\phi)t}u(C_t)\diff t.
\end{equation*}
The infinite horizon problem discussed in Section \ref{sec:singleagent} is therefore the same as the one solved by agents subject to mortality if we set $\rho=\tilde{\rho}+\phi$.

If agents have wealth processes that are independent copies of $W^\ast$, then the stationary distribution of wealth in the economy is the distribution of $W^\ast_T$, where $T$ is independent of $W^\ast$. This is the essential insight of \citet{reed2001}. From our discussion in Section \ref{sec:singleagent} we know that $W^\ast$ is a Markov-modulated \Levy process parametrized by $\varpi$, $\Pi$, $\Psi(z)$ and $\Upsilon(z)$, where $\Psi(z)$ is a diagonal matrix with $n$th diagonal entry $(y_n-b_n)z$, and $\Upsilon(z)$ has all entries equal to one. The corresponding matrix $\A(z)$ in \eqref{eq:defA} has diagonal entries $\A_{nn}(z)=(y_n-b_n)z+\pi_{nn}-\phi$ and off-diagonal entries $\A_{nn'}(z)=\pi_{nn'}$. In view of Propositions \ref{prop:convex} and \ref{prop:finiteT} and Remark \ref{rem:existence}, if $b_n<y_n$ for some $n$, then the equation $\zeta(\A(s))=0$ admits a unique positive solution $s=\alpha$. A similar statement applies to the existence of a unique negative solution. The following result is therefore an immediate consequence of Theorem \ref{thm:tailprob}.

\begin{prop}\label{prop:partialeq}
	If $b_n<y_n$ and $b_{n'}>y_{n'}$ for some $n,n'$, then the equation $\zeta(\A(s))=0$ has unique positive and negative solutions $s=\alpha,-\beta$, and the upper and lower tails of the stationary distribution of wealth decay exponentially at rates $\alpha$ and $\beta$ in the sense of \eqref{eq:mainup} and \eqref{eq:mainlow}.
\end{prop}

The preceding analysis may be regarded as one of partial equilibrium because the interest rate $r$ was treated as exogenous. In a general equilibrium model, the interest rate needs to adjust so that the asset market clears. Changes in $r$ affect the system of equations \eqref{eq:bnsystem} determining $b_1,\dots,b_N$ in Proposition \ref{prop:optimum}. We therefore write $b_n=b_n(r)$, and set $b(r)=(b_1(r),\dots,b_N(r))^\top$ and $y=(y_1,\dots,y_N)^\top$. Since an individual agent's rate of saving or borrowing in state $n$ is $y_n-b_n$, and the proportion of agents in state $n$ is $\varpi_n$, the condition for general equilibrium is
\begin{equation}\label{eq:geneq}
	\varpi^\top(y-b(r))=0.
\end{equation}
Note that under condition \eqref{eq:geneq} the expected slope of $W^\ast$ at any given point in time is zero, implying that aggregate wealth -- that is, $\E(W^\ast_T)$, the expected value of a random draw from the stationary distribution of wealth -- is zero. Consequently, the net flow of assets returned by departing agents is zero.

Our final result, proved in Appendix \ref{subsec:proofapplication}, shows that general equilibrium exists, and that in general equilibrium the condition in Proposition \ref{prop:partialeq} for exponential tail decay is automatically satisfied when income fluctuates.
\begin{prop}\label{prop:geneq}
	There exists an interest rate $r>0$ such that \eqref{eq:geneq} is satisfied. For any such $r$, if $y_1,\dots,y_N$ take at least two distinct values, then $b_n(r)<y_n$ and $b_{n'}(r)>y_{n'}$ for some $n,n'$. Consequently, the tails of the stationary distribution of wealth decay exponentially as in Proposition \ref{prop:partialeq}.
\end{prop}
Note that Proposition \ref{prop:geneq} guarantees the existence of general equilibrium, but says nothing about uniqueness. \citet{Toda2017} shows that, in a discrete-time version of our model, there is a unique equilibrium if income follows a first-order autoregressive law of motion as in \citet{Wang2003}, but multiple equilibria may obtain under a number of other common time series models for income.

\citet{Moll2020} analyze the stationary distribution of wealth in a continuous-time heterogeneous-agent model in which, as in our model, the sole source of heterogeneity is bounded fluctuation in income. They establish that the support of the stationary distribution of wealth in their model is bounded from above. This is an instance of a more general principle established by \citet{StachurskiToda2019} in a discrete-time context: for a wide class of heterogeneous-agent models in which income fluctuation is the sole source of heterogeneity, the stationary distribution of wealth inherits the upper tail behavior (in this instance, boundedness) of income shocks. Contrary to this principle, we have established that the stationary distribution of wealth in our model has exponential (hence unbounded) upper tail even though income has finite (hence bounded) support. The critical feature of our model leading to this discrepancy is that agents have CARA utility, violating a bounded relative risk aversion condition imposed by \citet{StachurskiToda2019} and \citet{Moll2020}. Under CARA, the optimal consumption path obtained here in continuous-time and by \citet{Wang2003} in discrete-time leads to a nonstationary wealth path. We are thus forced to introduce random birth/death as in \citet{Yaari1965} so as to obtain a stationary distribution of wealth. Under bounded relative risk aversion, \citet{StachurskiToda2019} and \citet{Moll2020} obtain a stationary distribution of wealth even with infinitely lived agents.

The technical machinery developed in Section \ref{sec:results} of this article may be applied more generally to models in which the wealth path of an agent subject to random birth/death may be written as an increasing function of a Markov-modulated \Levy process. In a companion article, \citet{BeareToda2020} study a discrete-time model in which agents are exposed to fluctuations in capital income. Owing to the multiplicative nature of capital income, it turns out to be the logarithm of an agent's wealth path which is (the discrete-time counterpart to) a Markov-modulated \Levy process. The effect of the logarithm is to generate a stationary distribution of wealth with a Pareto upper tail. In the simpler economic model considered in this article, the absence of investment risk leads to a stationary distribution of wealth with an exponential upper tail, which may be less desirable from an empirical perspective.

\appendix

\section{Mathematical appendix}

\subsection{Exponential tails via the Laplace transform}\label{subsec:nakagawa}

A theorem of \citet{Nakagawa2007} is the key tool we use to prove Theorems \ref{thm:tailprob} and \ref{thm:exptail}. Nakagawa's theorem places bounds on tail probabilities depending on singularities of the associated Laplace transform. The relevant singularities for our problem are simple poles. In this Appendix we present a simplified version of Nakagawa's theorem that is sufficient to handle simple poles.

First we briefly review the Laplace transform. For more details see \cite*{Widder1941} and \citet[ch.~7]{Lukacs1970}. Given a real random variable $X$ with cumulative distribution function (CDF) $F(x)$, let $M(s)=\int_{-\infty}^\infty \e^{sx}\diff F(x)$ be the corresponding MGF, defined initially for all real $s$ such that the integral is finite. The set of all such $s$ forms an interval containing zero, possibly a singleton, which we denote $\I$. The right and left boundary points of this interval, which may be real or $\pm\infty$, are denoted $\alpha$ and $-\beta$ and called the right and left abscissae of convergence. We extend the domain of $M$ into the complex plane by setting $M(z)=\int_{-\infty}^\infty \e^{zx}\diff F(x)$ for all complex $z$ such that the integral is well-defined; this is precisely the set $\S=\set{z\in \C:\Re z\in\I}$. Using the dominated convergence theorem, it is easy to see that $M(z)$ is holomorphic on the interior of $\S$, which is called the strip of holomorphicity (Figure \ref{fig:tauberian}). The lines $\Re z=\alpha,-\beta$ comprising the boundary of $\S$ are called the right and left axes of convergence. Viewed as a function of a complex variable, the MGF $M(z)$ is called the Laplace transform of the CDF $F(x)$, or of the random variable $X$.

\begin{figure}
	\centering
	\begin{tikzpicture}[scale=4]
	\draw[black, thick, ->] (-1,0) -- (1,0);
	\draw[black, thick, ->] (0,-.5) -- (0,1.1);
	\node[below left] at (0,0) {$0$};
	\node[right,font=\small] at (0,1.1) {$\Im{z}$};
	\node[below,font=\small] at (1,0) {$\Re{z}$};
	\filldraw[blue] (0.8,0) circle (.75pt);
	\filldraw[blue] (-0.8,0) circle (.75pt);
	\draw[<-] (-.82,-.05) -- (-.9,-.25);
	\node[below left,font=\small, align=left] at (-.9,-.25) {Left abscissa of\\
		convergence};
	\draw[<-] (.82,-.05) -- (.9,-.25);
	\node[below right,font=\small, align=left] at (.9,-.25) {Right abscissa of\\
		convergence};
	\draw[blue, thick, <->] (-.8,-.5) -- (-.8,1.1);
	\node[left,font=\small] at (-.8,.55) {\rotatebox{90}{Left axis of convergence}};
	\draw[blue, thick, <->] (.8,-.5) -- (.8,1.1);
	\node[right,font=\small] at (.8,.55) {\rotatebox{-90}{Right axis of convergence}};
	\fill [blue,opacity=.1] (-.8,1.05) rectangle (.8,-.45);
	\node[blue, font=\small, align=left] at (-.4,.4) {Strip of\\holomorphicity};
	\end{tikzpicture}
	\caption{Strip of holomorphicity of the Laplace transform.}\label{fig:tauberian}
\end{figure}

There is a close relationship between the tail probabilities of a CDF and the abscissae of convergence of its Laplace transform. In general, we have
\begin{equation*}
\limsup_{x\to\infty}\frac{1}{x}\log\Pr(X>x)=-\alpha,\quad\limsup_{x\to \infty}\frac{1}{x}\log\Pr(X< -x)=-\beta\label{classictaub}
\end{equation*}
whenever the relevant abscissa is nonzero (\citealp[pp.~42--43, 241]{Widder1941}; \citealp[p.~194]{Lukacs1970}). Moreover, each abscissa is a singularity of the Laplace transform \citep[p.~58]{Widder1941}. Nakagawa's theorem establishes a tighter relationship between the abscissae of convergence and the tail probabilities that depends on the nature of the singularities at the abscissae, and the location of other singularities along the axes of convergence. For our purposes, it will be enough to consider the case where the singularities at the abscissae are simple poles. We state a result providing bounds for the right tail probabilities; an analogous result holds for the left tail.

\begin{thm}\label{thm:tauberian}
	Let $X$ be a real random variable and $M(z)$ its Laplace transform, with right abscissa of convergence $\alpha\in (0,\infty)$. Suppose that $M(z)$ can be meromorphically extended to an open set containing its right axis of convergence, with a simple pole at $\alpha$. Let $-C<0$ denote the residue of $M(z)$ at $\alpha$, and let $B>0$ denote the supremum of all $b>0$ such that $\alpha$ is the unique singularity of $M(z)$ on $\set{\alpha+it:t\in(-b,b)}$. Then if $B<\infty$ we have
	\begin{equation*}\label{eq:GVbound}
	\frac{2\pi C/B}{\e^{2\pi \alpha/B}-1}\le \liminf_{x\to\infty}\e^{\alpha x}\Pr(X>x)\le \limsup_{x\to\infty}\e^{\alpha x}\Pr(X>x)\le \frac{2\pi C/B}{1-\e^{-2\pi \alpha/B}},
	\end{equation*}
	and if $B=\infty$ we have
	\begin{equation*}\label{eq:GVboundUnique}
	\lim_{x\to\infty}\e^{\alpha x}\Pr(X>x)=\frac{C}{\alpha}.
	\end{equation*}
\end{thm}
\begin{proof}
	The result follows from Theorem 5* of \citet{Nakagawa2007}, which provides more complicated bounds that obtain when the singularity at the right abscissa of convergence is a pole of arbitrary order.
\end{proof}

\subsection{Simple poles of matrix-valued functions}\label{subsec:schumacher}

In this appendix we state a result providing conditions under which a singularity of a matrix-valued function of a complex variable is a simple pole, and characterizing the residue of that pole. It is used to prove Theorems \ref{thm:tailprob} and \ref{thm:exptail}.

\begin{thm}\label{thm:matrixpencil}
	Let $A(z)$ be an $N\times N$ matrix-valued holomorphic function of $z\in\Omega$, where $\Omega$ is some open and connected subset of the complex plane. Suppose that $A(z)$ is invertible for some $z\in\Omega$. Then $A(z)$ has a meromorphic inverse $A(z)^{-1}$ on $\Omega$, with poles at the points of noninvertibility of $A(z)$. If $A(z_0)$ has rank $r<N$ for some $z_0\in\Omega$, so that $z_0$ is a pole of $A(z)^{-1}$, then the following conditions are equivalent.
	\begin{enumerate}
		\item\label{item:pencil1} $z_0$ is a simple pole of $A(z)^{-1}$.
		\item\label{item:pencil2} The zero eigenvalue of $A(z_0)$ has equal geometric and algebraic multiplicities.
		\item\label{item:pencil3} The $r\times r$ matrix $y^\top A'(z_0)x$ is invertible, where $x$ and $y$ can be any $N\times r$ matrices of full column rank such that $A(z_0)x=0$ and $y^\top A(z_0)=0$. 
	\end{enumerate}
	Under any of these equivalent conditions, the residue of $A(z)^{-1}$ at the simple pole $z_0$ is equal to $x(y^\top A'(z_0)x)^{-1}y^\top$.
\end{thm}

\begin{proof}
Meromorphicity of $A(z)^{-1}$ when $A(z)$ is somewhere invertible was proved by \citet{Steinberg1968}. The equivalence of \ref{item:pencil1} and \ref{item:pencil2} was proved by \citet*{Howland1971}. The equivalence of \ref{item:pencil1} and \ref{item:pencil3} and the residue formula were proved by \citet*{Schumacher1986}; see also \citet*[pp.~562--563]{Schumacher1991}.
\end{proof}

\subsection{Spectral properties of Metzler matrices}\label{subsec:metzler}

A real square matrix is said to be Metzler if its off-diagonal entries are nonnegative. In this Appendix we present three results on Metzler matrices that are used repeatedly in Appendices \ref{subsec:proofresults} and \ref{subsec:proofapplication}. The first result shows that the well-known Perron-Frobenius theorem for nonnegative square matrices extends to Metzler matrices if we shift our focus from the spectral radius to the spectral abscissa.
\begin{thm}\label{thm:PFmetzler}
	Let $A$ be a Metzler matrix. Then $\zeta(A)$ is an eigenvalue of $A$, and there are corresponding left and right eigenvectors with nonnegative entries. If, in addition, $A$ is irreducible, then $\zeta(A)$ is an algebraically simple eigenvalue of $A$, and there are corresponding left and right eigenvectors with strictly positive entries.
\end{thm}
\begin{proof}
	The result may be proved by applying the Perron-Frobenius theorem to the matrix $A+dI$, where $I$ is an identity matrix and $d\ge 0$ is chosen large enough to make $A+dI$ nonnegative. See, for example, Corollary 3.2 of \citet[ch.~4]{Smith1995}.
\end{proof}

The next result concerns Metzler matrices with constant row sums.
\begin{thm}\label{thm:rowsums}
If all row sums of a Metzler matrix $A$ are equal, then $\zeta(A)$ is equal to the common row sum.
\end{thm}
\begin{proof}
	The result may be proved by applying Theorem 8.1.22 of \citet{HornJohnson2013} to $A+dI$, where $d\ge 0$ is chosen large enough to make $A+dI$ nonnegative.
\end{proof}

We also require a monotonicity property of the spectral abscissa.
\begin{thm}\label{thm:strictmonotonicity}
	Let $A=(a_{nn'})$ be a Metzler matrix, and let $B=(b_{nn'})$ be a complex matrix of the same size. Suppose that the entries of $A$ and $B$ satisfy
	\begin{equation}\label{eq:DeutschIneq}
	\text{$\Re b_{nn}\le a_{nn}$ for all $n$, and $\abs{b_{nn'}}\le a_{nn'}$ for all $n\ne n'$.}
	\end{equation}
	Then $\zeta(B)\le\zeta(A)$. If, in addition, $A$ is irreducible and at least one of the inequalities in \eqref{eq:DeutschIneq} is strict, then $\zeta(B)<\zeta(A)$.
\end{thm}
\begin{proof}
	Let $C=(c_{nn'})$ be the Metzler matrix with diagonal entries $c_{nn}=\Re b_{nn}$ and off-diagonal entries $c_{nn'}=\abs{b_{nn'}}$. The weak inequality $\zeta(B)\le\zeta(C)$ was established in Corollary 1 of \citet{Deutsch1975}. Since $C\le A$, we may obtain $\zeta(C)\le\zeta(A)$ by applying Corollary 8.1.9 of \citet{HornJohnson2013} to the matrices $A+dI$ and $C+dI$, where $d\ge 0$ is chosen large enough to make both matrices nonnegative. When $A$ is irreducible and at least one of the inequalities in \eqref{eq:DeutschIneq} holds strictly, we may obtain $\zeta(C)<\zeta(A)$ by applying Theorem 8.4.5 of \citet{HornJohnson2013} in the same fashion to show that if $\zeta(C)=\zeta(A)$ then $C=A$.
\end{proof}

\subsection{Proofs of results in Section \ref{sec:results}}\label{subsec:proofresults}

\begin{proof}[Proof of Proposition \ref{prop:convex}]
	The off-diagonal entries of $\A(s)$ are log-convex functions of $s$ (since each is a positively scaled MGF), while the diagonal entries of $\A(s)$ are convex functions of $s$ (since each is a log-MGF shifted by a constant). Theorem 1.1 of \cite{Nussbaum1986} therefore implies that $\zeta(\A(s))$ is a convex function of $s$.
	
	Since $\A(0)\le\Pi$, Theorem \ref{thm:strictmonotonicity} implies that $\zeta(\A(0))\le\zeta(\Pi)$. Theorem \ref{thm:rowsums} implies that $\zeta(\Pi)=0$. Thus $\zeta(\A(0))\le 0$. The last sentence of the proposition is an obvious consequence of the convexity of $\zeta(\A(s))$.
\end{proof}

\begin{proof}[Proof of Proposition \ref{prop:finiteT}]
	It will be convenient to augment our state space with an absorbing state into which we transition at time $T$. To this end, let $\widetilde{\mathcal{N}}=\set{1,\dots,N+1}$, and for $t\ge 0$ let
	\begin{equation*}	
	\widetilde{J}_t=\begin{cases*}
	J_t& if $t<T$\\
	N+1& if $t\ge T$.
	\end{cases*}
	\end{equation*}
	The process $\widetilde{J}=(\widetilde{J}_t)_{t\ge 0}$ is a time-homogeneous Markov process with state space $\widetilde{\mathcal{N}}$. The $(N+1)\times1$ vector of initial state probabilities $\widetilde{\varpi}$ and $(N+1)\times(N+1)$ infinitesimal generator matrix $\widetilde{\Pi}$ for $\widetilde{J}$ are given by
	\begin{equation*}	
	\widetilde{\varpi}=\begin{bmatrix}\varpi\\0\end{bmatrix},\quad\widetilde{\Pi}=\begin{bmatrix}
	\Pi-\Phi&-(\Pi-\Phi)1_N\\
	0_N^\top&0\end{bmatrix}=\begin{bmatrix}
	\A(0)&-\A(0)1_N\\0_N^\top&0
	\end{bmatrix},
	\end{equation*}
	where $0_N$ (respectively, $1_N$) denotes an $N\times1$ vector with each entry equal to zero (respectively, one). For $n\in\mathcal{N}$, let $p_n=\Pr(T<\infty\mid J_0=n)$, the $n$th absorption probability for state $N+1$. (We may set $p_n=1$ for any state $n$ such that $P(J_0=n)=0$.) Define the vector $p=(p_1,\dots,p_N)^\top$. Clearly $\Pr(T<\infty)=1$ if and only if $p=1_N$. We will show that the conditions $p=1_N$ and $\zeta(\A(0))<0$ are both equivalent to $\A(0)$ being invertible.
	
	We first show that $p=1_N$ if and only if $\A(0)$ is invertible. Theorem 3.3.1 of \citet{Norris1997} provides the following characterization of absorption probabilities: the system of linear equations
	\begin{equation}\label{eq:Norris}
	\A(0)x=\A(0)1_N
	\end{equation}
	admits a minimal nonnegative solution, which is $x=p$. If $\A(0)$ is invertible then $x=1_N$ is the unique solution to \eqref{eq:Norris}, and so $p=1_N$. If $\A(0)$ is not invertible then the set of solutions to \eqref{eq:Norris} forms a hyperplane, so that the minimal nonnegative solution must have at least one entry equal to zero, implying that $p\ne 1_N$.
	
	We next show that $\A(0)$ is invertible if and only if $\zeta(\A(0))<0$. Since $\A(0)$ is Metzler, we know that its spectral abscissa is an eigenvalue (Theorem \ref{thm:PFmetzler}), and so it suffices to show that $\zeta(\A(0))\le 0$. This was established in Proposition \ref{prop:convex}.
\end{proof}

Our proofs of Proposition \ref{prop:Sminus} and Theorem \ref{thm:exptail} rely on the following lemma, which resembles a well-known result on characteristic functions.

\begin{lem}\label{lem:ridge}
	Let $X$ be a real random variable and $M(z)=\E(\e^{zX})$ its Laplace transform, with right abscissa of convergence $\alpha>0$. Then for any $s\in[0,\alpha)$ we have $\abs{M(s+it)}\le M(s)$ for all real $t$, with strict inequality for all real $t\ne 0$ if and only if the distribution of $X$ is not supported on a lattice.
\end{lem}
\begin{proof}
	The inequality $\abs{M(s+it)}\le M(s)$ is part of Theorem 7.1.2 of \citet{Lukacs1970}, and the remainder of the lemma follows from an obvious modification to the proof of Theorem 2.1.4 of \citet{Lukacs1970}.
\end{proof}

\begin{proof}[Proof of Proposition \ref{prop:Sminus}]
	Fix $z\in\S_-$, and let $s=\Re z\in\I_-$. The diagonal entries of $\A(z)$ and $\A(s)$ are $\A_{nn}(z)=\psi_n(z)+\pi_{nn}-\phi_n$ and $\A_{nn}(s)=\psi_n(s)+\pi_{nn}-\phi_n$, and thus by Lemma \ref{lem:ridge} satisfy $\abs{\exp\A_{nn}(z)}\le\exp\A_{nn}(s)$, which we may rewrite as
	$\Re\A_{nn}(z)\le\A_{nn}(s)$. The off-diagonal entries of $\A(z)$ and $\A(s)$ are $\A_{nn'}(z)=\pi_{nn'}\upsilon_{nn'}(z)$ and $\A_{nn'}(s)=\pi_{nn'}\upsilon_{nn'}(s)$, and thus by Lemma \ref{lem:ridge} satisfy $\abs{\A_{nn'}(z)}\le\A_{nn'}(s)$. Since $\A(s)$ is Metzler, we deduce from Theorem \ref{thm:strictmonotonicity} that $\zeta(\A(z))\le\zeta(\A(s))$. We conclude by observing that $\zeta(\A(s))<0$ due to the definition of $\I_-$.
\end{proof}

\begin{proof}[Proof of Proposition \ref{prop:MGF}]
	As in the proof of Proposition \ref{prop:finiteT}, we employ the augmented state space $\widetilde{\mathcal{N}}$. For $t\ge 0$, let $\widetilde{W}_t=W_{T\wedge t}$, and note that
	\begin{equation}\label{eq:MGF1}
	\e^{z\widetilde{W}_t}1(\widetilde{J}_t=N+1)=\e^{zW_T}1(T\le t).
	\end{equation}
	The process $\widetilde{W}=(\widetilde{W}_t)_{t\ge 0}$ is a Markov-modulated \Levy process with Markov-modulator $\widetilde{J}$. The law of $\widetilde{W}$ conditional on $\widetilde{J}$ is parametrized by the $(N+1)\times(N+1)$ matrices
	\begin{equation*}	
	\widetilde{\Psi}(z)=\begin{bmatrix}
	\Psi(z)&0_N\\
	0_N^\top&0\end{bmatrix},\quad
	\widetilde{\Upsilon}(z)=\begin{bmatrix}
	\Upsilon(z)&1_N\\
	1_N^\top&1\end{bmatrix}.
	\end{equation*}
	Fix $n\in\mathcal{N}$ such that $\Pr(J_0=n)>0$. By \eqref{eq:MGF1} and Proposition \ref{prop:asmussen}, for any $t>0$, the conditional expectation $\E(\e^{zW_T}1(T\le t)\mid \widetilde{J}_0=n)$ is equal to the entry in row $n$, column $N+1$ of the matrix
	\begin{equation*}	
	\exp(t(\widetilde{\Psi}(z)+\widetilde{\Pi}\odot\widetilde{\Upsilon}(z)))=\exp\left(t\begin{bmatrix}\A(z)&-\A(0)1_N\\0_N^\top&0\end{bmatrix}\right).
	\end{equation*}
	Applying Lemma 10.5.1 of \citet{ChenFrancis1995}, which provides a formula for the exponential of a block upper-triangular matrix, and using the fact that $\A(z)$ is invertible for $z\in\S_-$ (Proposition \ref{prop:Sminus}), we find that $\E(\e^{zW_T}1(T\le t)\mid \widetilde{J}_0=n)$ is equal to the entry in row $n$ of the vector
	\begin{align*}
	-\int_0^t\exp(s\A(z))\A(0)1_N\diff s=-\A(z)^{-1}(\exp(t\A(z))-I)\A(0)1_N.
	\end{align*}
	Since $\zeta(\A(z))<0$ for all $z\in\S_-$ (Proposition \ref{prop:Sminus}), we have $\exp(t\A(z))\to 0$ as $t\to\infty$, and so $\E(\e^{zW_T}1(T\le t)\mid \widetilde{J}_0=n)$ converges to the entry in row $n$ of the vector $\A(z)^{-1}\A(0)1_N$ as $t\to\infty$.
	
	It remains only to show that
	\begin{equation}\label{eq:MGF2}
	\E(\e^{zW_T}1(T\le t)\mid J_0=n)\to\E(\e^{zW_T}\mid J_0=n)\text{ as }t\to\infty.
	\end{equation}
	Since $\Pr(T<\infty)=1$, we have $\Pr(\e^{zW_T}1(T\le t)\to\e^{zW_T}\mid J_0=n)=1$. For real $z=s\in\I_-$, we may thus obtain \eqref{eq:MGF2} by applying the monotone convergence theorem. In view of what was shown in the previous paragraph, this means that $\E(\e^{sW_T}\mid J_0=n)$ is the entry in row $n$ of the vector $\A(s)^{-1}\A(0)1_N$, which is finite. We may thus obtain \eqref{eq:MGF2} for any complex $z\in\S_-$ by applying the dominated convergence theorem, using $\e^{(\Re z)W_T}$ as the dominating function.
\end{proof}

\begin{proof}[Proof of Theorem \ref{thm:tailprob}]
	We will consider the case where the equation $\zeta(\A(s))=0$ admits a positive solution $s=\alpha$ in the interior of $\I$. The case of a negative solution may be handled by a symmetric argument.
	
	Let $\Omega$ denote the interior of $\S$, an open and connected subset of the complex plane, nonempty since it contains $\alpha$. The matrix-valued function $\A(z)$ is holomorphic on $\Omega$. Since $\zeta(\A(s))$ is a convex function of $s\in\I$ (Proposition \ref{prop:convex}), with $\zeta(\A(0))<0$ (Proposition \ref{prop:finiteT}) and $\zeta(\A(\alpha))=0$, it must be the case that $\A(z)$ is invertible at all real numbers between $0$ and $\alpha$, which are elements of $\Omega$.  It follows from Theorem \ref{thm:matrixpencil} that $\A(z)$ has a meromorphic inverse $\A(z)^{-1}$ on $\Omega$, with poles at the points of noninvertibility of $\A(z)$. Since $\A(\alpha)$ is Metzler, we know that its spectral abscissa of zero is an eigenvalue (Theorem \ref{thm:PFmetzler}). Thus $\A(z)$ is not invertible at $\alpha\in\Omega$, and $\alpha$ is a pole of $\A(z)^{-1}$.
	
	We now show that $\alpha$ is a simple pole of $\A(z)^{-1}$, and determine the associated residue. Since $\A(\alpha)$ is irreducible Metzler, it follows from Theorem \ref{thm:PFmetzler} that the zero eigenvalue of $\A(\alpha)$ is algebraically simple, hence geometrically simple, and is associated with left and right eigenvectors $x,y$ with strictly positive entries. From Theorem \ref{thm:matrixpencil} we thus deduce that $\alpha$ is a simple pole of $\A(z)^{-1}$, with residue given by
	$$R = x (y^\top \A'(\alpha)x)^{-1}y^\top = cxy^\top,$$
	where $c=(y^\top \A'(\alpha)x)^{-1}$ is a nonzero real number.
	
	By Proposition \ref{prop:MGF}, the MGF of $W_T$ is given by $M_{W_T}(z)=\varpi^\top\A(z)^{-1}\A(0)1_N$ for $z\in\S_-$. Moreover, since $\A(z)^{-1}$ is meromorphic on $\Omega$, this equation defines a meromorphic extension of $M_{W_T}(z)$ to $\Omega$. As $z\to\alpha$, we obtain
	\begin{equation*}
	(z-\alpha)M_{W_T}(z)=\varpi^\top (z-\alpha)\A(z)^{-1}\A(0)1_N \to \varpi^\top cxy^\top\A(0)1_N=c(\varpi^\top x)(y^\top \A(0)1_N).
	\end{equation*}
	The row sums of $\A(0)$ are bounded by those of $\Pi$, which are zero. Since $\A(0)$ is invertible, not all rows sums are zero. Thus $\A(0)1_N$ has nonpositive entries, with at least one strictly negative entry. The right eigenvector $y$ has strictly positive entries, so $y^\top\A(0)1_N<0$. Similarly, $\varpi$ has nonnegative entries, with at least one strictly positive entry, and the left eigenvector $x$ has strictly positive entries, so $\varpi^\top x>0$. As noted above, $c\ne 0$. Therefore $\lim_{z\to\alpha}(z-\alpha)M_{W_T}(z)\ne 0$. This shows that $\alpha$ is a simple pole of $M_{W_T}(z)$. It now follows from Theorem \ref{thm:tauberian} that the limits inferior and superior of $\e^{\alpha w}\Pr(W_T>w)$ as $w\to\infty$ are strictly positive and finite.
\end{proof}

\begin{proof}[Proof of Theorem \ref{thm:exptail}]
	It suffices for us to show that $\alpha$ is the unique singularity of $M_{W_T}(z)$ on its axis of convergence, because in this case the application of Theorem \ref{thm:tauberian} at the end of the proof of Theorem \ref{thm:tailprob} establishes that the limits inferior and superior of $\e^{\alpha w}\Pr(W_T>w)$ as $w\to\infty$ are equal. We will do this by showing that $\zeta(\A(\alpha+it))<0$ for all real $t\ne 0$. The diagonal entries of $\A(\alpha+it)$ and $\A(\alpha)$ are $\A_{nn}(\alpha+it)=\psi_n(\alpha+it)+\pi_{nn}-\phi_n$ and $\A_{nn}(\alpha)=\psi_n(\alpha)+\pi_{nn}-\phi_n$, and thus by Lemma \ref{lem:ridge} satisfy $\abs{\exp\A_{nn}(\alpha+it)}\le\exp\A_{nn}(\alpha)$, which we may rewrite as
	\begin{equation}\label{eq:Adiagineq}
	\Re\A_{nn}(\alpha+it)\le\A_{nn}(\alpha).
	\end{equation}
	The off-diagonal entries of $\A(\alpha+it)$ and $\A(\alpha)$ are $\A_{nn'}(\alpha+it)=\pi_{nn'}\upsilon_{nn'}(\alpha+it)$ and $\A_{nn'}(\alpha)=\pi_{nn'}\upsilon_{nn'}(\alpha)$, and thus by Lemma \ref{lem:ridge} satisfy
	\begin{equation}\label{eq:Aoffdiagineq}
	\abs{\A_{nn'}(\alpha+it)}\le\A_{nn'}(\alpha).
	\end{equation}
	Moreover, under our non-lattice condition, Lemma \ref{lem:ridge} implies that at least one of the inequalities in \eqref{eq:Adiagineq} or \eqref{eq:Aoffdiagineq} must hold strictly for all real $t\ne 0$. Noting that $\A(\alpha)$ is irreducible Metzler, we see that $\A(\alpha+it)$ and $\A(\alpha)$ satisfy the conditions placed upon $B$ and $A$ in Theorem \ref{thm:strictmonotonicity}, and deduce that $\zeta(\A(\alpha+it))<\zeta(\A(\alpha))=0$ for all real $t\ne 0$.
\end{proof}

\subsection{Proofs of results in Section \ref{sec:application}}\label{subsec:proofapplication}

\begin{lem}\label{lem:fixedpoint}
	The system of equations \eqref{eq:bnsystem} has a unique solution $(b_1,\dots,b_N)\in\R^N$. The solution varies continuously with $r>0$, with $\min_nb_n\to\infty$ as $r$ decreases to zero.
\end{lem}

\begin{proof}
	
	Fix a constant $k>0$, and let $f:\R^N\to\R^N$ be the map with $n$th entry
    \begin{equation}\label{eq:fixedpoint}
    	f_n(b)=\frac{1}{k+1}\left(kb_n+y_n-\frac{1}{\gamma}+\frac{\rho}{\gamma r}-\frac{1}{\gamma r}\sum_{n'=1}^N\pi_{nn'}\e^{\gamma(b_n-b_{n'})}\right).
   	\end{equation}
    The system of equations \eqref{eq:bnsystem} may be rewritten as $b=f(b)$. By Corollary 2.1.1 of \citet{Zhang2013}, if we can find an order interval $[u,v]\subset\R^N$ such that $f$ is increasing and concave on $[u,v]$ with $f_n(u)>u_n$ and $f_n(v)<v_n$ for all $n$, then the restriction of $f$ to $[u,v]$ has a unique fixed point. To this end, we fix $\eta>0$ and set
	\begin{equation*}    
	u=\left(\min_ny_n-\frac{1}{\gamma}+\frac{\rho}{\gamma r}-\eta\right)1_N,\quad v=\left(\max_ny_n-\frac{1}{\gamma}+\frac{\rho}{\gamma r}+\eta\right)1_N.
	\end{equation*}
    It is easily verified that $f_n(u)>u_n$ and $f_n(v)<v_n$ for all $n$. Observe that
    \begin{equation*}
    	\frac{\partial f_m(b)}{\partial b_n}=\begin{cases*}
    		\frac{1}{k+1}\left(k-\frac{1}{r}\sum_{n'\ne n}\pi_{nn'}\e^{\gamma(b_n-b_{n'})}\right)& if $m=n$,\\\
    		\frac{1}{(k+1)r}\pi_{mn}\e^{\gamma(b_m-b_n)}& if $m\ne n$.
    	\end{cases*}
    \end{equation*}
    Since $\Pi$ is Metzler and $[u,v]$ is bounded, we may choose $k$ large enough that all partial derivatives are nonnegative on $[u,v]$, ensuring that $f$ is increasing on $[u,v]$. Next, observe that, for any $\delta\in(0,1)$ and $b,b'\in\R^N$, the Metzler property of $\Pi$ and convexity of $x\mapsto\e^{\gamma x}$ together imply that
	\begin{equation*}    
	\pi_{nn'}\e^{\gamma\left((\delta b_n+(1-\delta)b'_n)-(\delta b_{n'}+(1-\delta)b'_{n'})\right)}\le \delta\pi_{nn'}\e^{\gamma(b_n-b_{n'})}+(1-\delta)\pi_{nn'}\e^{\gamma(b'_n-b'_{n'})}
	\end{equation*}
    for all $n,n'$, with equality when $n=n'$. It thus follows from \eqref{eq:fixedpoint} that $f$ is concave. We deduce from the aforementioned result of \citet{Zhang2013} that the restriction of $f$ to $[u,v]$ has a unique fixed point, which (with a mild abuse of notation) we denote $b(r)$. It follows that $b(r)$ is the unique solution to \eqref{eq:bnsystem} in $[u,v]$. Moreover, since $\eta>0$ was arbitrary, $b(r)$ is in fact the unique solution to \eqref{eq:bnsystem} in $\R^N$, and satisfies
    \begin{equation*}\label{eq:solutionbounds}
    	\min_{n'}y_{n'}-\frac{1}{\gamma}+\frac{\rho}{\gamma r}\le b_n(r)\le\max_{n'}y_{n'}-\frac{1}{\gamma}+\frac{\rho}{\gamma r}\quad\text{for all }n.
    \end{equation*}
    We therefore see that $\min_nb_n(r)\to\infty$ as $r$ decreases to zero. Furthermore, since $b(r)$ is the unique value of $b$ maximizing $-\norm{f(b)-b}$, which varies continuously with $b$ and $r$, the Berge maximum theorem implies that $b(r)$ is continuous in $r$.
\end{proof}

\begin{proof}[Proof of Proposition \ref{prop:optimum}]
	Let $\mathcal{D}$ be the collection of all functions $V:\mathcal{N}\times\R\to\R$ such that $V_n(w)$ is a continuously differentiable function of $w\in\R$ for each $n\in\mathcal{N}$. The optimization problem to be solved is a special case of Example (c) in Sections III.4 and III.7 of \citet{FlemingSoner2006}, with $x$, $z$ and $u$ in their notation corresponding to $W$, $J$ and $C$ in our notation. It is shown there that the relevant Hamilton-Jacobi-Bellman (HJB) equation for our optimization problem is
	\begin{equation}\label{eq:HJB}
		\rho V_n(w)=\max_{c\in\R}\set{u(c)+V_n'(w)(rw-c+y_n)+\sum_{n'=1}^N\pi_{nn'}V_{n'}(w)}.
	\end{equation}
	Note that this is an $N$-state generalization of the HJB equation obtained by \citet{Moll2020} in the two-state case. A function $V^\ast\in\mathcal{D}$ that solves \eqref{eq:HJB} for every $(n,w)\in\mathcal{N}\times\R$ is called a classical solution. With our CARA specification of $u(c)$, elementary calculus may be used to verify that the maximum in \eqref{eq:HJB} is achieved by
	\begin{equation}\label{eq:cstar}
		c=-\frac{1}{\gamma}\log V'_n(w),
	\end{equation}
	and that a classical solution to \eqref{eq:HJB} is given by $V^\ast_n(w)=-\e^{-\gamma(rw+b_n)}/\gamma r$, where $b_1,\dots,b_n$ are the unique constants given by Lemma \ref{lem:fixedpoint}.
	
	A consumption flow $C\in\mathcal{C}$ satisfies the no-Ponzi condition \eqref{eq:budget}, meaning that it belongs to $\mathcal{C}_1$, if and only if it satisfies the so-called transversality condition
	\begin{equation}\label{eq:transversality}
		\limsup_{t\to\infty}\e^{-\rho t}\E_0V^\ast_{J_t}(W^C_t)=0.
	\end{equation}
	By Theorem 9.1 of \citet[p.~140]{FlemingSoner2006}, if there is a consumption flow $C^\ast\in\mathcal{C}_1$ with associated wealth process $W^\ast$ such that
	\begin{equation}\label{eq:argmin}
		C^\ast_t\in\argmax_{c\in\R}\set{u(c)+V^{\ast\prime}_n(W^\ast_t)(rW^\ast_t-c+y_n)+\sum_{n'=1}^N\pi_{nn'}V^\ast_{n'}(W^\ast_t)}\quad\text{while }J_t=n,
	\end{equation}
    then $U(C)$ attains its maximum over $\mathcal{C}_1$ at $C^\ast$. Moreover, any such $C^\ast$ is the unique maximizer of $U(C)$ over $\mathcal{C}_1$ due to the strict concavity of $u(c)$.
    
    We will show that the consumption flow $C^\ast$ given by \eqref{eq:optimalc} belongs to $\mathcal{C}_1$ and satisfies \eqref{eq:argmin}. It is obvious that $C^\ast\in\mathcal{C}$, so to show that $C^\ast\in\mathcal{C}_1$, we need to verify the no-Ponzi condition \eqref{eq:budget}. Since $W^\ast$ is a Markov-modulated \Levy process, it follows from Proposition \ref{prop:asmussen} that, when $J_0=n$, the quantity $\e^{-\rho t}\E_0(\e^{-\gamma rW_t^\ast})$ appearing in \eqref{eq:budget} is the $n$th row sum of $\exp(tA)$, where
	\begin{equation*}
	A=-\rho I-\gamma r\diag(y-b)+\Pi.
	\end{equation*}
    Let $D$ be the $N\times N$ diagonal matrix with $n$th diagonal entry $\e^{\gamma b_n}$. From \eqref{eq:bnsystem} we see that $DAD^{-1}$ is Metzler with each row summing to $-r$. Consequently, Theorem \ref{thm:rowsums} implies that $\zeta(DAD^{-1})=-r$. Since $A$ and $DAD^{-1}$ have the same spectra, we deduce that $\zeta(A)=-r$, which implies that $\exp(tA)\to 0$ as $t\to\infty$. Thus \eqref{eq:budget} is satisfied and $C^\ast\in\mathcal{C}_1$. Finally, in view of the fact that \eqref{eq:cstar} solves the maximization in \eqref{eq:HJB}, and that $C_t^\ast=-(1/\gamma)\log V^{\ast\prime}_n(W_t^\ast)$ while $J_t=n$, it is clear that $C^\ast$ satisfies \eqref{eq:argmin}. We conclude that $C^\ast$ uniquely maximizes $U(C)$ over $\mathcal{C}_1$. 
\end{proof}
\begin{proof}[Proof of Proposition \ref{prop:geneq}]
For $r>0$, let $g(r)=\varpi^\top(y-b(r))$. The general equilibrium condition \eqref{eq:geneq} is satisfied by any $r>0$ such that $g(r)=0$. Lemma \ref{lem:fixedpoint} implies that $g(r)$ is a continuous function of $r>0$ with $g(r)\to-\infty$ as $r$ decreases to zero. Therefore, if we can show that 
$g$ is somewhere nonnegative, then the existence of $r>0$ satisfying $g(r)=0$ will follow from the intermediate value theorem. Subtracting $y_n$ from either side of \eqref{eq:bnsystem}, multiplying by $-\gamma r\varpi_n$, and summing over $n$, we obtain
\begin{equation*}
\gamma rg(r)=r-\rho+\sum_{n=1}^N\sum_{n'=1}^N\varpi_n\pi_{nn'}\e^{\gamma(b_n(r)-b_{n'}(r))}.
\end{equation*}
The Metzler property of $\Pi$ and the inequality $\e^x\ge1+x$ together imply that 
\begin{equation*}
\pi_{nn'}\e^{\gamma(b_n(r)-b_{n'}(r))}\ge\pi_{nn'}(1+\gamma(b_n(r)-b_{n'}(r)))
\end{equation*}
for all $n,n'$, with equality when $n=n'$. We therefore have
\begin{equation*}
\gamma rg(r)\ge r-\rho+\sum_{n=1}^N\varpi_n(1+\gamma b_n(r))\left(\sum_{n'=1}^N\pi_{nn'}\right)-\gamma\sum_{n'=1}^N\left(\sum_{n=1}^N\varpi_n\pi_{nn'}\right)b_{n'}(r).
\end{equation*}
Since $\sum_{n'}\pi_{nn'}=0$ and $\sum_{n}\varpi_n\pi_{nn'}=0$ (because $\varpi$ is the stationary distribution corresponding to $\Pi$), it follows that $\gamma rg(r)\ge r-\rho$. Thus $g(r)>0$ for $r>\rho$, and we conclude that there exists $r\in (0,\rho]$ such that $g(r)=0$.

When $r$ satisfies \eqref{eq:geneq}, it must be the case that either $b_n(r)<y_n$ and $b_{n'}(r)>y_{n'}$ for some $n,n'$, or $b(r)=y$. We will show that if $b(r)=y$ then all the values $y_1,\dots,y_N$ are equal. Substituting $y$ for $b$ in the system of equations \eqref{eq:bnsystem}, we obtain
\begin{equation*}
(\rho-r)\e^{-\gamma y_n}=\sum_{n'=1}^N\pi_{nn'}\e^{-\gamma y_{n'}}
\end{equation*}
for each $n$. Thus $v=(\e^{-\gamma y_1},\dots,\e^{-\gamma y_N})^\top$ is a right eigenvector of $\Pi$ with corresponding eigenvalue $\rho-r$. The stationary distribution $\varpi$ is a left eigenvector of $\Pi$ with corresponding eigenvalue zero, so we have $0=\varpi^\top\Pi v=(\rho-r)\varpi^\top v$, implying that $r=\rho$. Thus the eigenvalue corresponding to the right eigenvector $v$ is zero. Since $\Pi$ is irreducible Metzler with $\zeta(\Pi)=0$ (by Theorem \ref{thm:rowsums}), Theorem \ref{thm:PFmetzler} implies that there is a unique (up to scalar multiplication) right eigenvector of $\Pi$ corresponding to its zero eigenvalue. All row sums of $\Pi$ are zero, so one such right eigenvector is $1_N$. Thus $v$ must be a scalar multiple of $1_N$, meaning that $y_1=\dots=y_N$.
\end{proof}

\bibliographystyle{plainnat}

\end{document}